\documentclass[12pt]{amsart}
\usepackage{amsmath,amsthm,amssymb}

\theoremstyle{definition}
\newtheorem{dfn}{Definition}[section]

\newtheorem{rem}[dfn]{Remark}

\theoremstyle{plain}
\newtheorem{thm}[dfn]{Theorem}
\newtheorem{prop}[dfn]{Proposition}

\newtheorem{lem}[dfn]{Lemma}

\newcommand{\orderedprod}{\mathop{\overrightarrow{\prod}}}

\title{Some coefficients of rank 2 cluster scattering diagrams}
\author{Ryota Akagi}
\address{Graduate School of Mathematics\\ nagoya University\\Chikusa-ku\\ Nagoya\\464-8602\\ Japan.}
\email{ryota.akagi.e6@math.nagoya-u.ac.jp}

\begin{document}
\begin{abstract}
The purpose of this paper is to translate the expression of rank 2 cluster scattering diagrams via dilogarithm elements into via formal power series. As a corollary, we prove some conjectures introduced by Thomas Elgin, Nathan Reading, and Salvatore Stella.
\end{abstract}
\maketitle
\section{Introduction}
\subsection{Background}
Scattering diagrams were introduced by \cite{KS06, GS11} in the study of mirror symmetry. Later, in \cite{GHKK18}, the application for cluster algebra theory, which was introduced by \cite{FZ02}, was found. They showed the existence of scattering diagrams having much information of cluster algebras. These scattering diagrams are called {\em cluster scattering diagrams}. They introduced the {\em theta basis} of cluster algebras by using cluster scattering diagrams, and this construction showed some significantly important conjectures in \cite{FZ07}. 
\par
Cluster scattering diagrams consist of walls, which is a pair of a support and an automorphism. Its support structure is closely related to the {\em $g$-vector fan} in the cluster algebra theory. The walls in the $g$-vector fan are well understood by the mutation of cluster algebras. However, explicit behavior outside of $g$-vector fan is so complicated. In \cite{GL25}, they showed that all walls outside the $g$-vector fan are nontrivial in rank 2 cluster scattering diagrams. In \cite{ERS24}, they introduced some conjectures about properties of walls.
\par
In this paper, we focus on the cluster scattering diagrams of rank 2. In this case, the theta basis can be established by another combinatorial framework called the {\em greedy basis}, which was introduced by \cite{LLZ14}. In \cite{CGMMRSW17}, they showed that these two bases coincide with each other. Recently, in \cite{BLM25}, they obtained a combinatorial formula for the walls based on this correspondence, and they introduced a new combinatorial framework called tight grading.
\par
There are another construction of rank 2 cluster scattering diagram independent of the property of theta basis. The construction of cluster scattering diagrams are closely related to a dilogarithm identity \cite{Nak23}. By using this method, we may obtain some properties for walls \cite{Aka25}.
\par
The purpose of this paper is to translate the results in \cite{Aka25} written by a dilogarithm identity into another notations by formal power series. Based on it, we show some conjectures given by \cite{ERS24}.
\begin{rem}
After submitting this paper, most conjectures in \cite{ERS24} are solved in \cite{BLMMR25} by using the quive moduli and tight gradings. Although their results cover the ones in this paper, their method is different.
\end{rem}
\subsection{Main results}
The main object in this paper is the wall function $f_{{\bf n}_0}$, which is a formal power series of two variables $\hat{y}_1,\hat{y}_2$. (See Definition~\ref{def: wall function}.) Then, we may give a restriction for the expression of wall functions.
\begin{thm}[Theorem~\ref{thm: expression of wall function}]
Consider the cluster scattering diagram $\mathfrak{D}_{b,c}$ with the initial exchange matrix
$B=\left(\begin{smallmatrix}
0 & c\\
-b & 0
\end{smallmatrix}\right)$ \textup{($b,c \in \mathbb{Z}_{\geq 1}$)}. Then, for each ${\bf n}_0=(n_1,n_2) \in \mathbb{Z}^{2}_{\geq 0}\setminus\{(0,0)\}$ with $\gcd(n_1,n_2)=1$, the wall function $f_{{\bf n}_0}$ corresponding to the wall orthogonal to ${\bf n}_0$ may be expressed as
\begin{equation}
f_{{\bf n}_0}=\left\{\prod_{k=1}^{\infty}(1+\hat{\bf y}^{{k {\bf n}_0}})^{U_{k{\bf n}_0}(c,b)}
\right\}^{g({\bf n}_0;b,c)},
\end{equation}
where $g({\bf n}_0;b,c)=\frac{\gcd(n_1b,n_2,c)}{\gcd(n_1,n_2)}$ and $U_{k{\bf n}_0}(c,b) \in \mathbb{Q}_{\geq 0}$. Moreover, if we view $b$ and $c$ as indeterminates, $U_{k{\bf n}_0}(c,b)$ is a polynomial for each $k \in \mathbb{Z}_{\geq 1}$ and ${\bf n}_0$, and it has the maximum degree $\deg_{(c,b)} U_{k{\bf n}_0}(c,b)=(kn_1-1,kn_2-1)$.
\end{thm}
In \cite{ERS24}, their conjectures are given by the expanded form of wall functions. We may give a formula to translate the dilogarithm elements into formal power series (Lemma~\ref{lem: u to tau}). Based on this, we may solve some conjectures in \cite{ERS24}.
\begin{prop}[Proposition~\ref{prop: 1,2,3}, \ref{prop: 5,6}, \ref{prop: 14,15,16,17,18}]
Conjecture~1, 2, 3, 5, 6, 11, 15 in \cite{ERS24} are true. Partially, we may obtain similar facts to Conjecture~14, 16, 17, 18 in \cite{ERS24}.
\end{prop}
\begin{rem}
After submitting the paper of the first version, the author was informed by Amanda Burcroff, Kyungyong Lee, and Lang Mou that we can show Conjecture~11 (Proposition~\ref{prop: 11}) completely. The proof of Proposition~\ref{prop: 11} is given by them.
\end{rem}

\subsection{Structure of this paper}
In Section~\ref{sec: Definition of scattering diagrams}, we recall the definition of cluster scattering diagrams. In Section~\ref{sec: definition by dilogarithm identity}, we use the notation of dilogarithm elements, and in Section~\ref{sec: formal power series}, we use the notation of formal power series. These notations follow from \cite{Nak23}.
\par
In Section~\ref{sec: translation}, we introduce a formula to translate a dilogarithm identity into the formal power series.
\par
In Section~\ref{sec: proof of conjectures}, we prove some conjectures in \cite{ERS24}.

\subsection*{Acknowledgement}
I thank Amanda Burcroff, Kyungyong Lee, Lang Mou, Tomoki Nakanishi, Nathan Reading, and Salvatore Stella for the useful comments and advices.
In particular, I thank Amanda Burcroff, Kyungyong Lee, and Lang Mou for suggesting the proof of Proposition~\ref{prop: 11}.
\par
This work was supported by JSPS KAKENHI Grant Number JP25KJ1438.

\section{Definition}\label{sec: Definition of scattering diagrams}
In this section, we define the cluster scattering diagrams with two forms. In Section~\ref{sec: definition by dilogarithm identity}, we define it via the dilogarithm elements, and in Section~\ref{sec: formal power series}, we define it via formal power series called wall functions.
\subsection{From dilogarithm identities}\label{sec: definition by dilogarithm identity}
In this paper, we fix an initial exchange matrix
\begin{equation}
B=\left(\begin{matrix}
0 & c\\
-b & 0
\end{matrix}\right) \quad (b,c \in \mathbb{Z}_{>0})
\end{equation}
and set its skew-symmetrizer
\begin{equation}
D=\left(\begin{matrix}
c^{-1} & 0\\
0 & b^{-1}
\end{matrix}\right).
\end{equation}
We write ${\bf e}_1=(1,0)^{\top}$ and ${\bf e}_2=(0,1)^{\top}$.
\par
Based on this, we define the following data (called {\em fixed data}):
\begin{itemize}
\item $N=\mathbb{Z}^2$, $N^{\circ}=(c\mathbb{Z})\oplus(b\mathbb{Z}) \subset N$, and $N^{+}=\mathbb{Z}_{\geq 0}^2\backslash\{(0,0)\} \subset N$.
\item Skew-symmetric bilinear form $\{\,,\,\}$ on $N$ defined by 
\begin{equation}
\{a_1{\bf e}_1+a_2{\bf e}_2,b_1{\bf e}_1+b_2{\bf e}_2\}=a_1b_2-a_2b_1.
\end{equation}
\item $M=(c\mathbb{Z})\otimes(b\mathbb{Z})$, $M^{\circ}=\mathbb{Z}^2 \supset M$ (we often see them as a dual space of $N$ and $N^{\circ}$ with respect to the following inner product), and $M_{\mathbb{R}}=\mathbb{R}^2 \supset M^{\circ}$. For any ${\bf m} \in M_{\mathbb{R}}$ and ${\bf n} \in N$, we define a symmetric bilinear form
\begin{equation}
\langle {\bf m},{\bf n} \rangle= {\bf m}^{\top}D{\bf n}.
\end{equation}
\end{itemize}
Another important basic object is the {\em structure group} $G$. For the definition, we need to construct a Lie algebra structure, but we do not have to know the detail in this paper. So, we skip the construction of $G$ here, and we summarize the property we need to use later. (See \cite[\S~III.1]{Nak23})
\par
The structure group $G$ consists of elements of the following form:
\begin{equation}
\exp\left(\sum_{{\bf n} \in N^{+}}s_{\bf n}X_{\bf n}\right) \quad (s_{\bf n} \in \mathbb{Q}),
\end{equation}
where $X_{\bf n}$ (${\bf n} \in N^{+}$) are indeterminates.
We may admit some (not all) infinite products of the elements in $G$.  The following elements play important roles in the cluster scattering diagrams.
\begin{dfn}
For each ${\bf n} \in N^{+}$, the following element is called the {\em dilogarithm element}.
\begin{equation}
\Psi[{\bf n}]=\exp\left(\sum_{j=1}^{\infty}\frac{(-1)^{j+1}}{j^2}X_{j{\bf n}}\right) \quad (n \in N^{+}).
\end{equation}
For each ${\bf n} \in N^{+}$ and $s \in \mathbb{Q}$, we set
\begin{equation}
\Psi[{\bf n}]^{s}=\exp\left(s\sum_{j=1}^{\infty}\frac{(-1)^{j+1}}{j^2}X_{j{\bf n}}\right).
\end{equation}
\end{dfn}
For each ${\bf n}=(n_1,n_2)^{\top} \in N^{+}$, let $\deg({\bf n})=n_1+n_2$. We need the following properties of the structure group. (In fact, (c) and (d) comes from the properties of the definition of $G$.)
\begin{prop}[{\cite{Nak23}}]\label{prop: structure group}
The structure group $G$ satisfies the following properties. 
\\
\textup{(a) \cite[Prop.~III.1.13]{Nak23}} The set $\{\Psi[{\bf n}]^{s} \mid {\bf n} \in N^{+}, s \in \mathbb{Q}\}$ is a generator of $G$ in the following meaning: every element of $G$ may be expressed as a product of (possibly infinitely many) $\Psi[{\bf n}]^{s}$.
\\
\textup{(b) \cite[(III.1.34)]{Nak23}} For each $l =1,2,\dots$, let $G^{>l}$ be the subgroup of $G$ generated by $\{\Psi[{\bf n}]^{s} \mid \deg({\bf n})>l, s \in \mathbb{Q}\}$. Then, $G^{>l}$ is a normal subgroup of $G$. So, we may consider its quotient $G^{\leq l}=G/G^{>l}$ with the projection $\pi_{l}:G \to G^{\leq l}$.
\\
\textup{(c)} Consider a formal infinite product $P=\prod\Psi[{\bf n}]^s$. (We do not assume this product is well-defined in $G$.) If $\pi_{l}(P)$ becomes the finite product for any $l=1,2,\dots$ by ignoring $\mathrm{id} \in G^{\leq l}$, then this infinite product $P$ is well-defined in $G$.
\\
\textup{(d)} For any (admitting infinite) products $P,P'$, the equality $P=P'$ holds if and only if $\pi_{l}(P)=\pi_{l}(P')$ for any $l=1,2,\dots$.
\end{prop}

\begin{prop}[{\cite[Prop.~III.1.14]{Nak23}}]\label{prop: fundamental relation}
For any ${\bf n},{\bf n}' \in N^{+}$, the following equalities hold.
\\
\textup{(a)} If $\{{\bf n}',{\bf n}\}=0$, then for any $s,s' \in \mathbb{Q}$, we have
\begin{equation}
\Psi[{\bf n}']^{s'}\Psi[{\bf n}]^{s}=\Psi[{\bf n}]^{s}\Psi[{\bf n}']^{s'}.
\end{equation}
\textup{(b)\ ({\em Pentagon relation})} If $\{{\bf n}',{\bf n}\} = s \neq 0$, then we have
\begin{equation}
\Psi[{\bf n}']^{\frac{1}{s}}\Psi[{\bf n}]^{\frac{1}{s}}=\Psi[{\bf n}]^{\frac{1}{s}}\Psi[{\bf n}+{\bf n}']^{\frac{1}{s}}\Psi[{\bf n}']^{\frac{1}{s}}.
\end{equation}
\end{prop}
These two relations may be seen as a fundamental relation to construct a rank 2 cluster scattering diagram. (However, it is not a fundamental relation of $G$.)
\begin{dfn}
Let $\Lambda$ be a totally ordered set, and let $\{{\bf n}_{\lambda}\}_{\lambda \in \Lambda} \subset N^{+}$ and $\{s_{\lambda}\}_{\lambda \in \Lambda} \subset \mathbb{Q}$. Then, a product of dilogarithm elements
\begin{equation}
P=\prod_{\lambda \in \Lambda}\Psi[{\bf n}_{\lambda}]^{s_\lambda}
\end{equation}
is called an {\em ordered product} if it satisfies the following conditions:
\begin{itemize}
\item If $\lambda < \lambda'$, then $\{{\bf n}_{\lambda},{\bf n}_{\lambda'}\} \geq 0$.
\item If $\lambda < \lambda'$ and $\{{\bf n}_{\lambda},{\bf n}_{\lambda'}\}=0$, then $\deg({\bf n}_\lambda)<\deg({\bf n}_{\lambda'})$.
\end{itemize}
For any $\{s_{\bf n}\}_{{\bf n} \in N^{+}} \subset \mathbb{Q}$, we write
\begin{equation}
\orderedprod_{{\bf n} \in N^{+}} \Psi[{\bf n}]^{s_{\bf n}}
\end{equation}
as the ordered product of $\{\Psi[{\bf n}]^{s_{\bf n}}\}$, respectively.
\end{dfn}
The second condition of the definition of ordered product is not essential due to Proposition~\ref{prop: fundamental relation}~(a). However, we assume this condition for the uniqueness of this product.
\par
Note that every ordered product is always well defined by Proposition~\ref{prop: structure group}~(c), and it is uniquely determined by Proposition~\ref{prop: structure group}~(d).
\begin{dfn}
For any ${\bf n}=(n_1,n_2) \in N^{+}$, we define the {\em normalized factor} $\delta({\bf n})>0$ of ${\bf n}$ as the smallest positive rational number such that $\delta({\bf n}){\bf n} \in N^{\circ}$. In this setting, it is given by
\begin{equation}
\delta({\bf n})=\frac{bc}{\gcd(n_1b,n_2c)}.
\end{equation}
\end{dfn}
\begin{prop}[{\cite[Prop.~III.5.4]{Nak23}} Ordering lemma]\label{prop: ordering lemmma}
Consider a finite product of the form $P=\prod \Psi[{\bf n}]^{\delta({\bf n})s_{\bf n}}$ with $s_{\bf n} \in \mathbb{Z}_{>0}$. Then, there exists a unique ordered product
\begin{equation}
\orderedprod_{{\bf n} \in N^{+}} \Psi[{\bf n}]^{\delta({\bf n})s'_{\bf n}} \quad (s'_{\bf n} \in \mathbb{Z}_{\geq 0})
\end{equation}
which is the same as $P$ in $G$. Moreover, $P$ may be rearranged to this ordered product by using the relations in Proposition~\ref{prop: fundamental relation} (possibly infinitely many times).
\end{prop}
Based on dilogatithm elements, we define scattering diagrams.
\begin{dfn}
Let
\begin{equation}
N^{+}_{\mathrm{pr}}=\{(n_1,n_2) \in N^{+} \mid \gcd(n_1,n_2)=1\}.
\end{equation}
For each ${\bf n}_0 \in N^{+}_{\mathrm{pr}}$, we define $G_{{\bf n}_0}^{||}$ as the subgroup of $G$ generated by $\{\Psi[k{\bf n}_0]^{s} \mid k=1,2,\dots, s \in \mathbb{Q}\}$, and we call it the {\em parallel subgroup} of ${\bf n}_0$.
\end{dfn}
Note that every parallel subgroup is a commutative group because of Proposition~\ref{prop: fundamental relation}~(a).
\begin{dfn}\label{dfn: wall}
We call the following triplet $w=(\mathfrak{d},g)_{{\bf n}_0}$ a {\em wall}.
\begin{itemize}
\item ${\bf n}_0 \in N^{+}_{\mathrm{pr}}$.
\item $g \in G_{{\bf n}_0}^{||}$.
\item $\mathfrak{d}=\mathbb{R}_{\geq 0}{\bf m} \subset M_{\mathbb{R}}$ for some ${\bf m} \in M_{\mathbb{R}}\setminus\{{\bf 0}\}$ such that $\langle {\bf m}, {\bf n}_0\rangle=0$.
\end{itemize}
We call ${\bf n}_0$ as a {\em normal vector} of this wall. Since we may recover ${\bf n}_0$ from $\mathfrak{d}$, we sometimes omit ${\bf n}_0$.
\end{dfn}
\begin{dfn}[Scattering diagram]
A set of walls $\mathfrak{D}=\{w_{\lambda}=(\mathfrak{d}_{\lambda},g_{\lambda})_{{\bf n}_{\lambda}} \mid \lambda \in \Lambda\}$ is called a {\em scattering diagram} if $\#\{\lambda \in \Lambda \mid \pi_{l}(g_{\lambda}) \neq \mathrm{id}\} < \infty$ holds for any $l \in \mathbb{Z}_{>0}$.
We define its {\em support} $|\mathfrak{D}|=\bigcup_{\lambda} \mathfrak{d}_{\lambda} \subset M_{\mathbb{R}}$.
\end{dfn}
\begin{dfn}[Path ordered product]
Let $\mathfrak{D}$ be a scattering diagram. A continuous path $\gamma:[0,1] \to M_{\mathbb{R}}$ is said to be {\em admissible} if it satisfies the following:
\begin{itemize}
\item $\gamma(0),\gamma(1) \notin |\mathfrak{D}|$.
\item $\gamma:(0,1) \to M_{\mathbb{R}}$ is differentiable. 
\item If $\gamma$ and $|\mathfrak{D}|$ has an intersection, it is not the origin and $\gamma$ intersects $|\mathfrak{D}|$ transversely. Moreover, for each intersection $\gamma(t)$ at the intersection time $t$, suppose that it intersects to a wall $(\mathfrak{d},g)_{{\bf n}_0}$. Then, we define its {\em intersection sign} $\epsilon_{t}$ as follows:
\begin{equation}
\epsilon_{t}=\begin{cases}
1 & \langle \gamma'(t),{\bf n}_0 \rangle<0,\\
-1 & \langle \gamma'(t),{\bf n}_0 \rangle>0.
\end{cases}
\end{equation}
\end{itemize}
Let $\gamma$ be an admissible path, and let $\{t_{\lambda}\}_{\lambda \in \Lambda}$ be the collection of intersection time labeled by an ordered set $\Lambda$ such that $t_{\lambda} \leq t_{\lambda'}$ whenever $\lambda < \lambda'$. (If $\gamma$ intersects some walls at the same time, we put the same time in this collection.) Then, the {\em path ordered product} $\mathfrak{p}_{\gamma,\mathfrak{D}}$ is defined by
\begin{equation}
\mathfrak{p}_{\gamma,\mathfrak{D}}=\prod_{\lambda \in \Lambda} g_{\lambda}^{\epsilon_{t_\lambda}}.
\end{equation}
\end{dfn}
\begin{dfn}
A scattering diagram $\mathfrak{D}$ is said to be {\em consistent} if the path ordered products $\mathfrak{p}_{\gamma,\mathfrak{D}}$ are determined by its starting point $\gamma(0)$ and its ending point $\gamma(1)$ for any admissible path $\gamma$.
\end{dfn}
\begin{dfn}
Two scattering diagrams $\mathfrak{D}, \mathfrak{D}'$ are said to be {\em equivalent} if two path ordered product $\mathfrak{p}_{\gamma,\mathfrak{D}}, \mathfrak{p}_{\gamma,\mathfrak{D}'}$ are the same for any admissible path $\gamma$ for both $\mathfrak{D}$ and $\mathfrak{D}'$.
\end{dfn}

\begin{thm}\label{thm: scattering diagrams for dilogarithm identities}
For each ${\bf n} \in N^{+}$, a scattering diagram
\begin{equation}
\mathfrak{D}_{b,c}=\{w_{\bf n}=(\mathfrak{d}_{\bf n},g_{\bf n}) \mid {\bf n} \in N^{+}\} \cup \{(\mathbb{R}(-{\bf e}_2),\Psi[{\bf e}_1]^{c}),(\mathbb{R}{\bf e}_1,\Psi[{\bf e}_2]^{b})\}.
\end{equation}
which satisfies the following conditions exists uniquely.
\begin{itemize}
\item For each ${\bf n} \in N^{+}$, $g_{\bf n}=\Psi[{\bf n}]^{s_{\bf n}}$ for some $s_{\bf n} \in \mathbb{Q}$.
\item For each ${\bf n}=(n_1,n_2) \in N^{+}$, $\mathfrak{d}_{\bf n}=\mathbb{R}_{\geq 0}{\bf m}$ with ${\bf m}=(-cn_2,bn_1)$.
\end{itemize}
\end{thm}
\begin{dfn}
The scattering diagram equivalent to the one $\mathfrak{D}_{b,c}$ in Theorem~\ref{thm: scattering diagrams for dilogarithm identities} is called the {\em cluster scattering diagram}.
\end{dfn}
In rank 2 case, the construction of cluster scattering diagram is closely related to Proposition~\ref{prop: ordering lemmma}. Let $u_{\bf n}(c,b)$ be the rational number satisfying the following equality:
\begin{equation}
\Psi[{\bf e}_{2}]^{b}\Psi[{\bf e}_1]^c=\orderedprod_{{\bf n} \in N^{+}}\Psi[{\bf n}]^{u_{\bf n}(c,b)}.
\end{equation}
(By Proposition~\ref{prop: ordering lemmma}, this is uniquely determined and $u_{\bf n}(c,b)$ satisfies $u_{\bf n}(c,b) \in \delta({\bf n})\mathbb{Z}_{\geq 0}$.)
Then, the cluster scattering diagram $\mathfrak{D}_{b,c}$ may be obtained by
\begin{equation}\label{eq: cluster scattering diagram via dilogarithm elements}
\{(\mathfrak{d}_{\bf n},\Psi[{\bf n}]^{u_{\bf n}(c,b)}) \mid {\bf n} \in N^{+}\} \cup \{(\mathbb{R}(-{\bf e}_2),\Psi[{\bf e}_1]^{c}),(\mathbb{R}{\bf e}_1,\Psi[{\bf e}_2]^{b})\}.
\end{equation}
In this expression, we have already known some properties in \cite{Aka25}. To state them, we introduce the {\em binomial coefficients} $\binom{x}{k}$ ($x \in \mathbb{Q}$, $k \in \mathbb{Z}_{\geq 0}$) as
\begin{equation}
\binom{x}{k}=\frac{x(x-1)\cdots(x-(k-1))}{k!}.
\end{equation}
When we view $x$ as an indeterminate, it is a polynomial with $\deg_{x}\binom{x}{k}=k$.
\begin{prop}[{\cite[Thm.~7.1, Prop.~8.6]{Aka25}}]\label{prop: exponent of dilogarithm identity}
Here, we view $b,c$ as indeterminates.
For each ${\bf n}=(n_1,n_2) \in N^{+}$, $u_{\bf n}(c,b)$ is a polynomial in $b$ and $c$. Moreover, it may be expressed as
\begin{equation}
u_{\bf n}(c,b)=\frac{1}{\gcd(n_1,n_2)}\sum_{\substack{1 \leq i \leq n_1,\\
1 \leq j \leq n_2}} \alpha_{\bf n}(i,j) \binom{c}{i}\binom{b}{j}
\end{equation}
for some nonnegative integers $\alpha_{\bf n}(i,j) \in \mathbb{Z}_{\geq 0}$ independent of $b,c$.
\par
In particular, (since $\alpha_{\bf n}(n_1,n_2) \neq 0$) this polynomial $u_{\bf n}(c,b)$ has the maximum degree $\deg_{(c,b)} u_{\bf n}(c,b)=(n_1,n_2)$.
\end{prop}
\subsection{Formal power series expression}\label{sec: formal power series}
Let $x_1,x_2$ be indeterminates, and set $\hat{y}_1=x_2^{-b}$ and $\hat{y}_2=x_1^{c}$. Let $\mathbb{Q}[[x_1,x_2^{-1}]]$ (resp. $\mathbb{Q}[[\hat{y}_1,\hat{y}_2]]$) be the ring of formal power series with two variables $x_1,x_2^{-1}$ (resp. $\hat{y}_1,\hat{y}_2$). For each ${\bf m}=(m_1,m_2) \in M^{\circ}$ and ${\bf n}=(n_1,n_2) \in N^{+}$, we write ${\bf x}^{\bf m}=x_1^{m_1}x_2^{m_2}$ and $\hat{\bf y}^{\bf n}=\hat{y}_1^{n_1}\hat{y}_{2}^{n_2}$. Let $\mathbb{Q}[[\hat{\bf y}^{\bf n}]] \subset \mathbb{Q}[[\hat{y}_1,\hat{y}_2]]$ be the ring of formal power series with one variable $\hat{\bf y}^{\bf n}$.
\par
We introduce the following group associated with a formal power series.
\begin{dfn}
Let ${\bf n}=(n_1,n_2) \in N^{+}_{\mathrm{pr}}$. For any $f \in \mathbb{Q}[[\hat{\bf y}^{\bf n}]]\backslash\{0\}$, let $\mathfrak{p}_{f} \in \mathrm{Aut}(\mathbb{Q}[[x_1,x_2^{-1}]])$ be a $\mathbb{Q}$-algebra automorphism induced by
\begin{equation}
\mathfrak{p}_{f}({\bf x}^{\bf m})={\bf x}^{\bf m}\cdot f^{\langle{\bf m,}\delta({\bf n}){\bf n}\rangle} \quad ({\bf m} \in M^{\circ} \cap \mathbb{Q}[[x_1,x_2^{-1}]]).
\end{equation}
Let $\tilde{G}$ be a subgroup of $\mathrm{Aut}(\mathbb{Q}[[x_1,x_2^{-1}]])$ generated by $\{\mathfrak{p}_{f}\mid f \in \mathbb{Q}[[\hat{y}_1,\hat{y}_2]]\setminus\{0\}\}$.
\end{dfn}
The correspondence $f \mapsto \mathfrak{p}_{f}$ is injective. Moreover, by the following fact, we may identify the element of structure group $G$ with the formal power series.
\begin{lem}[{\cite[(III.4.22)]{Nak23}}]\label{lem: group homomorphism}
There exists an injective group homomorphism $\rho : G \to \tilde{G}$ satisfying
\begin{equation}
\rho(\Psi[{\bf n}]^{\delta({\bf n})s})=\mathfrak{p}_{(1+\hat{\bf y}^{\bf n})^{s}} \quad ({\bf n} \in N^{+}, s \in \mathbb{Q}).
\end{equation}
Moreover, for each ${\bf n} \in N^{+}_{\mathrm{pr}}$ and $f,f' \in \mathbb{Q}[[\hat{\bf y}^{\bf n}]]$, we have
\begin{equation}
\mathfrak{p}_{f}\mathfrak{p}_{f'}=\mathfrak{p}_{ff'}.
\end{equation}
In particular, for any $g \in G_{\bf n}^{||}$, there exists a unique $f \in \mathbb{Q}[[\hat{\bf y}^{\bf n}]]$ such that $\rho(g)=\mathfrak{p}_{f}$.
\end{lem}
Based on this correspondence, we introduce another notion of walls.
\begin{dfn}\label{def: wall function}
For each wall $(\mathfrak{d},g)_{{\bf n}_0}$ in the sense of Definition~\ref{dfn: wall}, we also write $[\mathfrak{d},f]_{{\bf n}_0}=(\mathfrak{d},g)_{{\bf n}_0}$, where $f \in \mathbb{Q}[[\hat{\bf y}^{{\bf n}_0}]]$ is determined by $\rho(g)=\mathfrak{p}_{f}$. This $f$ is called a {\em wall function} of $(\mathfrak{d},g)_{{\bf n}_0}$.
\end{dfn}
Now, we may rephrase the cluster scattering diagram in Theorem~\ref{thm: scattering diagrams for dilogarithm identities} as follows:
\begin{prop}\label{prop: wall function}
Consider the cluster scattering diagram in (\ref{eq: cluster scattering diagram via dilogarithm elements}). Set $\hat{u}_{\bf n}(c,b)=\delta({\bf n})^{-1}u_{\bf n}(c,b)$ and, for any ${\bf n}_0 \in N^{+}_{\mathrm{pr}}$,
\begin{equation}\label{eq: wall function}
f_{{\bf n}_0}=\prod_{k=1}^{\infty}(1+\hat{\bf y}^{k{\bf n}_0})^{\hat{u}_{k{\bf n}_0}(c,b)}.
\end{equation}
Then, the following scattering diagram is equivalent to the cluster scattering diagram $\mathfrak{D}_{b,c}$.
\begin{equation}
\{[\mathfrak{d}_{{\bf n}_0},f_{{\bf n}_0}]_{{\bf n}_0} \mid {\bf n}_0 \in N^{+}_{\mathrm{pr}}\} \cup \{[\mathbb{R}(-{\bf e}_2),1+\hat{y}_1]_{{\bf e}_1},[\mathbb{R}{\bf e}_1,1+\hat{y}_2]_{{\bf e}_2}\}.
\end{equation}
\end{prop}
\begin{proof}
By Lemma~\ref{lem: group homomorphism}, we have
\begin{equation}
\rho(\Psi[{\bf n}]^{u_{\bf n}(c,b)})=\rho(\Psi[{\bf n}]^{\delta({\bf n})\hat{u}_{\bf n}(c,b))})=\mathfrak{p}_{(1+\hat{\bf y}^{\bf n})^{\hat{u}_{\bf n}(c,b)}}.
\end{equation}
Thus, we have $(\mathfrak{d}_{\bf n},\Psi[{\bf n}]^{u_{\bf n}(c,b)})=[\mathfrak{d},(1+\hat{\bf y}^{\bf n})^{\hat{u}_{\bf n}(c,b)}]$. By combining the walls which has the same normal vector, we may obtain the claim.
\end{proof}
When we argue the cluster scattering diagram based on this expression, the factor
\begin{equation}
g({\bf n};b,c)=\frac{\gcd(n_1b,n_2c)}{\gcd(n_1,n_2)} \quad ({\bf n}=(n_1,n_2))
\end{equation}
is helpful. This factor appears in \cite{ERS24}.
By using this factor, we may express $\delta({\bf n})^{-1}=\frac{\gcd(n_1,n_2)}{bc}g({\bf n};b,c)$ and
\begin{equation}\label{eq: u hat}
\hat{u}_{\bf n}(c,b)=\gcd(n_1,n_2)g({\bf n};b,c)\frac{u_{\bf n}(c,b)}{bc}.
\end{equation}
Thus, Proposition~\ref{prop: exponent of dilogarithm identity} may be rephrased as follows.
\begin{thm}\label{thm: expression of wall function}
For each ${\bf n}_0=(n_1',n_2') \in N^{+}_{\mathrm{pr}}$, the wall function $f_{{\bf n}_0}$ may be expressed as
\begin{equation}\label{eq: factrized form of wall functions}
f_{{\bf n}_0}=\left\{\prod_{k=1}^{\infty}(1+\hat{\bf y}^{{k {\bf n}_0}})^{U_{k{\bf n}_0}(c,b)}
\right\}^{g({\bf n}_0;b,c)},
\end{equation}
where
\begin{equation}
U_{k{\bf n}_0}(c,b) = \frac{k}{bc}u_{k{\bf n}_0}(c,b) \in \mathbb{Q}_{\geq 0}.
\end{equation}
Moreover, these exponents $U_{k{\bf n}_0}(c,b)$ satisfy the following properties.
\\
\textup{(a)} For any ${\bf n}=(n_1,n_2) \in N^{+}$, we may express
\begin{equation}\label{eq: expression of U}
U_{{\bf n}}(c,b)=\frac{1}{bc}\sum_{\substack{1 \leq i \leq n_1,\\
1 \leq j \leq n_2}} \alpha_{\bf n}(i,j) \binom{c}{i}\binom{b}{j}
\end{equation}
for some nonnegative integers $\alpha_{\bf n}(i,j) \in \mathbb{Z}_{\geq 0}$ independent of $b$ and $c$.
\\
\textup{(b)} Now we view $b,c$ as indeterminates. Then, $U_{\bf n}(c,b)$ is a polynomial in $b,c$. Moreover, it has the maximum degree $\deg_{(c,b)}U_{\bf n}(c,b)=(n_1-1,n_2-1)$.
\end{thm}
\begin{proof}
Note that $g(k{\bf n}_0;b,c)=g({\bf n}_0;b,c)$ and $\gcd(kn_1',kn_2')=k$ hold for any $k=1,2,\dots$. Thus, by substituting
\begin{equation}
\hat{u}_{k{\bf n}_0}(c,b)\overset{(\ref{eq: u hat})}{=}kg(k{\bf n}_0;b,c)\frac{u_{k{\bf n}_0}(c,b)}{bc}=g({\bf n}_0;b,c)U_{k{\bf n}_0}(c,b)
\end{equation}
into (\ref{eq: wall function}), we may obtain (\ref{eq: factrized form of wall functions}). The claim (a) follows from Proposition~\ref{prop: exponent of dilogarithm identity}. The claim (b) may be shown by (a). (Note that $\frac{1}{bc}\binom{c}{i}\binom{b}{j}=\frac{1}{i!j!}(c-1)\cdots(c-i+1)(b-1)\cdots(b-j+1)$ is a polynomial when $i,j \geq 1$.)
\end{proof}

\section{Translation into the scattering coefficients}\label{sec: translation}
Our main purpose is to translate $u_{\bf n}(c,b)$ into the following coefficients.
\begin{dfn}
For each ${\bf n}_0$, let $f_{{\bf n}_0} \in \mathbb{Q}[[\hat{\bf y}^{{\bf n}_0}]]$ as in Proposition~\ref{prop: wall function}. Then, we write its coefficients as follows:
\begin{equation}
f_{{\bf n}_0}=1+\sum_{k=1}^{\infty}\tau^{(b,c)}(k{\bf n}_0)\hat{\bf y}^{k{\bf n}_0}.
\end{equation}
\end{dfn}
By a direct calculation, we may obtain the following equality.
\begin{lem}\label{lem: u to tau}
For any ${\bf n}_0 \in N^{+}_{\mathrm{pr}}$ and $k=1,2,3,\dots$, we have
\begin{equation}\label{eq: u to tau}
\tau^{(b,c)}(k{\bf n}_0)=\sum_{\substack{
s_1,s_2,\dots,s_k \geq 0,\\
s_1+2s_2+\cdots+ks_k=k}}
\prod_{j=1}^{k}\binom{g({\bf n}_0;b,c)U_{j{\bf n}_0}(c,b)}{s_j}.
\end{equation}
\end{lem}
\begin{proof}
For simplicity, we set $z=\hat{\bf y}^{{\bf n}_0}$. Then, by (\ref{eq: factrized form of wall functions}), we have
\begin{equation}
f_{{\bf n}_0}=\prod_{k=1}^{\infty}(1+z^k)^{g({\bf n}_0;b,c)U_{j{\bf n}_0}(c,b)}.
\end{equation}
Note that $(1+z^k)^{g({\bf n}_0;b,c)U_{j{\bf n}_0}(c,b)}$ may be expanded as
\begin{equation}
(1+z^k)^{g({\bf n}_0;b,c)U_{j{\bf n}_0}(c,b)}=\sum_{s=0}^{\infty}\binom{{g({\bf n}_0;b,c)U_{j{\bf n}_0}(c,b)}}{s}z^{ks}.
\end{equation}
Thus, we have
\begin{equation}
\begin{aligned}
f_{{\bf n}_0}&=\prod_{k=1}^{\infty}\left\{\sum_{s_k=0}^{\infty}\binom{{g({\bf n}_0;b,c)U_{j{\bf n}_0}(c,b)}}{s_k}z^{ks_k}\right\}\\
&=\sum_{s_1,s_2,s_3,\dots =0}^{\infty}\left\{\prod_{j=1}^{\infty}\binom{{g({\bf n}_0;b,c)U_{j{\bf n}_0}(c,b)}}{s_j}\right\}z^{s_1+2s_2+3s_3+\cdots}.
\end{aligned}
\end{equation}
Consider $\tau^{(b,c)}(k{\bf n}_0)$, which is the coefficient of $z^k$ in the above equality. A factor in the above sum contributes to the coefficients of $z^k$ if and only if $s_1+2s_2+3s_3+\cdots=k$. So, we can restrict the sum region by $s_{k+1}=s_{k+2}=\cdots=0$. (If at least one $s_{j}$ ($j \geq k+1$) is positive, $s_1+2s_2+\dots$ is larger than $k$.) Thus, we may obtain (\ref{eq: u to tau}). 
\end{proof}

\section{Proof of some conjectures in \cite{ERS24}}\label{sec: proof of conjectures}
We start to show some conjectures introduced by \cite{ERS24}.
\subsection{Conjecture~1, 2, 3, 5, 6}
\begin{prop}[{\cite[Conj.~1, 2, 3]{ERS24}}]\label{prop: 2}\label{prop: 1,2,3}
For any ${\bf n}=(n_1,n_2)^{\top} \in N^{+}$, set $g=g({\bf n};b,c)$.\\
\textup{(a)}\ $\tau^{(b,c)}({\bf n})$ may be expressed as a polynomial in $g$, $b$, and $c$.
\\
\textup{(b)}\ $\deg_{g}\tau^{(b,c)}({\bf n})=\gcd(n_1,n_2)$ and $\tau^{(b,c)}({\bf n})$ has a factor $g$.
\\
\textup{(c)}\ $\deg_{(b,c)}\tau^{(b,c)}({\bf n})=(n_2-1,n_1-1)$.
\end{prop}
\begin{proof}
The statement (a) follows from Lemma~\ref{lem: u to tau}. (Note that a product and a sum of some polynomials are also a polynomial.)
Set ${\bf n}=k{\bf n}_0$ for some ${\bf n}_0 \in N^{+}_{\mathrm{pr}}$ and $k=1,2,\cdots$. Then, $k=\gcd(n_1,n_2)$ holds. By (\ref{eq: u to tau}), $\tau^{(b,c)}({\bf n})$ is the sum of the following terms
\begin{equation}\label{eq: terms of tau}
\prod_{j=1}^{k}\binom{gU_{j{\bf n}_0}(c,b)}{s_j},
\end{equation}
where $s_j \geq 0$ and $s_1+2s_2+3s_3+\dots+ks_k=k$.
For each factor, since a binomial coefficient $\binom{x}{i}$ is a polynomial in $x$ with $\deg_x\binom{x}{i}=i$, we have
\begin{align}
\deg_{g}\prod_{j=1}^{k}\binom{gU_{j{\bf n}_0}(c,b)}{s_j}&=\sum_{j=1}^{k}s_j \leq k,
\end{align}
where the equality holds when $(s_1,s_2,\dots,s_k)=(k,0,\dots,0)$. Thus, $\deg_{g}\tau^{(b,c)}({\bf n})=\gcd(n_1,n_2)$. Moreover, each term (\ref{eq: terms of tau}) has a factor $g$. Thus, $\tau^{(b,c)}({\bf n})$ also has a factor $g$. Hence, (b) holds. Set ${\bf n}_0=(n_1',n_2')$. Note that $n_1=kn_1'$ and $n_2=kn_2'$ hold. Moreover, since $\deg_{b} U_{j{\bf n}_0}(c,b)=jn_2'-1$ by Theorem~\ref{thm: expression of wall function}, we have
\begin{equation}\label{eq: deg b}
\begin{aligned}
\deg_{b}\prod_{j=1}^{k}\binom{gU_{j{\bf n}_0}(c,b)}{s_j}&=\sum_{j=1}^{n}j(n'_2-1)s_j=n'_2\sum_{j=1}^{n}js_j-\sum_{j=1}^{n}j\\
&= n'_2k-\sum_{j=1}^{n}j\leq n_2 -1,
\end{aligned}
\end{equation}
and, by doing the same argument,
\begin{equation}\label{eq: deg c}
\deg_{c}\prod_{j=1}^{k}\binom{gU_{j{\bf n}_0}(c,b)}{s_j}\leq n_1-1.
\end{equation}
Moreover, the equalities of these two inequalities hold when $(s_1,\dots,s_k)=(0,0,\dots,1)$. Note that the equality of (\ref{eq: deg b}) and (\ref{eq: deg c}) happens at the same $(s_1,\dots,s_k)=(0,\dots,1)$, and since $U_{k{\bf n}_0}(c,b)$ has the maximum degree $c^{n_1-1}b^{n_2-1}$, the claim (c) holds.
\end{proof}
\begin{prop}[{\cite[Conj.~5, 6]{ERS24}}]\label{prop: 5,6}
For any $n_1,n_2 \in \mathbb{Z}_{\geq 0}$, we have
\begin{equation}
\tau^{(b,c)}(n_1,1)=\frac{g((n_1,1);b,c)}{c}\binom{c}{n_1},
\quad
\tau^{(b,c)}(1,n_2)=\frac{g((1,n_2);b,c)}{b}\binom{b}{n_2}.
\end{equation}
\end{prop}
\begin{proof}
It follows from Lemma~\ref{lem: u to tau} and \cite[Prop.~8.8]{Aka25}, that is, $u_{(1,n_2)}(c,b)=c\binom{b}{n_2}$ and $u_{(n_1,1)}(c,b)=b\binom{c}{n_1}$.
\end{proof}

\subsection{Conjecture~11}
\begin{prop}[{\cite[Conj.~11]{ERS24}}]\label{prop: 11}
For any ${\bf n}=(n_1,n_2) \in N^{+}$, $\tau^{(b,b)}({\bf n})$ is a polynomial in $b$ with $\deg_b \tau^{(b,b)}({\bf n})=n_1+n_2-1$ that expands positively in the basis $\{\binom{b}{0},\binom{b}{1},\binom{b}{2},\dots\}$.
\end{prop}
The following proof is proposed by Amanda Burcroff, Kyungyong Lee, and Lang Mou.
\begin{proof}
Set ${\bf n}=k{\bf n}_0$ with $k=\gcd(n_1,n_2)$ and ${\bf n}_0 \in N^{+}_{\mathrm{pr}}$.
When $b=c$, we have $g({\bf n};b,b)=b$. Thus, by (\ref{eq: u to tau}), we have
\begin{equation}\label{eq: tau for 11}
\tau^{(b,b)}({\bf n})=\sum_{\substack{
s_1,s_2,\dots,s_k \geq 0,\\
s_1+2s_2+\cdots+ks_k=k}}\prod_{j=1}^{k}\binom{bU_{j{\bf n}_0}(b,b)}{s_j}.
\end{equation}
Let ${\bf n}_0=(n'_1,n'_2)$. Since $\deg_b bU_{j{\bf n}_0}(b,b)=j(n'_1+n'_2)-1$, we have
\begin{equation}
\begin{aligned}
\deg_b\prod_{j=1}^{k}\binom{bU_{j{\bf n}_0}(b,b)}{s_j}&=\sum_{j}s_j\left\{j(n'_1+n'_2)-1\right\}\\
&=(n'_1+n'_2)\sum_{j=1}^{k}s_jj-\sum_{j=1}^{k}j =(n'_1+n'_2)k-\sum_{j}s_j\\
&\leq n_1+n_2-1.
\end{aligned}
\end{equation}
The maximum degree is taken at $(s_1,\dots,s_k)=(0,\dots,0,1)$. Thus, we have $\deg_b\tau^{(b,b)}({\bf n})=n_1+n_2-1$.
\par
Next, we show that $\tau^{(b,b)}({\bf n})$ may be expanded by $\{\binom{b}{0}, \binom{b}{1},\binom{b}{2},\dots\}$ with positive coefficients. By (\ref{eq: expression of U}), we have
\begin{equation}
\begin{aligned}
bU_{j{\bf n}_0}(b,b)&=\frac{1}{b}\sum_{\substack{1 \leq p \leq n_1, \\ 1 \leq q \leq n_2}}\alpha_{j{\bf n}_0}(p,q)\binom{b}{p}\binom{b}{q}\\
&=b\sum_{\substack{1 \leq p  \leq n_1, \\ 1 \leq q \leq n_2}}\frac{\alpha_{j{\bf n}_0}(p,q)}{pq}\binom{b-1}{p-1}\binom{b-1}{q-1}.
\end{aligned}
\end{equation}
By \cite[Prop.~2.7]{Aka25}, the following fact holds.
\begin{quote}
If a polynomial $f(x),g(x)$ may be expanded by $\{\binom{x}{l} \mid l \in \mathbb{Z}_{\geq 0} \}$ with positive coefficients, then its product $f(x)g(x)$ and its composition $f(g(x))$ also satisfy this property.
\end{quote}
We apply this fact by setting $x=b-1$. Then, we may show that $\binom{b-1}{p-1}\binom{b-1}{q-1}$ may be expanded by $\{\binom{b-1}{l} \mid l \in \mathbb{Z}_{\geq 0}\}$. Since $\alpha_{j{\bf n}_0}(p,q) \geq 0$, the sum of these polynomials is also expanded by $\{\binom{b-1}{l} \mid l \in \mathbb{Z}_{\geq 0} \}$ with positive coefficients. Namely, we may write
\begin{equation}
bU_{j{\bf n}_0}(b,b)=b\sum_{l}\beta_{l}\binom{b-1}{l}
\end{equation}
for some $\beta_l \in \mathbb{Q}_{\geq 0}$. Note that $b\binom{b-1}{l}=(l+1)\binom{b}{l+1}$. Thus, we have
\begin{equation}
bU_{j{\bf n}_0}(b,b)=\sum_{l}(l+1)\beta_l\binom{b}{l+1}.
\end{equation}
Thus, $bU_{j{\bf n}_0}(b,b)$ can be expanded by $\{\binom{b}{l} \mid l\}$ with positive coefficients. By (\ref{eq: tau for 11}), $\tau^{(b,b)}({\bf n})$ can be seen as the compositions and products of such polynomials. Thus, the claim holds.
\end{proof}

\subsection{Conjecture~15 and related results of Conjecture~14, 16,17,18}
Fix one ${\bf n}=(n_1,n_2)^{\top} \in N^{+}_{\mathrm{pr}}$ and set $g=g({\bf n};b,c)$. Since $\tau^{(b,c)}({\bf n})$ is a polynomial in $g$ (Proposition~\ref{prop: 2}) which has a factor $g$, every $\tau^{(b,c)}({\bf n})$ may be express
\begin{equation}
\tau^{(b,c)}({\bf n})=\sum_{k=1}^{\gcd(n_1,n_2)}\tau^{(b,c)}({\bf n};k)g^{k},
\end{equation}
where $\tau^{(b,c)}({\bf n};k)$ is a polynomial in $b,c$ and independent of $g$.
We may show the following properties.
\begin{prop}[Cf.~{\cite[Conj.~14-18]{ERS24}}]\label{prop: 14,15,16,17,18}
Let ${\bf n}=(n_1,n_2)^{\top} \in N^{+}_{\mathrm{pr}}$.
\\
\textup{(a)} For any $k=1,2,\dots,\gcd(n_1,n_2)$, we have
\begin{equation}
\deg_{b} \tau^{(b,c)}({\bf n};k) \leq n_2-k,\quad \deg_{c}\tau^{(b,c)}({\bf n};k) \leq n_1-k.
\end{equation}
\textup{(b)} If $\gcd(n_1,n_2)=1$, then, for any $k=1,2,\dots$, we have 
\begin{align}
\tau^{(b,c)}(k{\bf n};k)&=\frac{\tau^{(b,c)}({\bf n};1)^{k}}{k!},\label{eq: 15}\\
\tau^{(b,c)}(k{\bf n};k-1)&=\frac{\tau^{(b,c)}({\bf n};1)^{k-2}}{(k-2)!}p_{\bf n},\label{eq: 16-18}
\end{align}
where $p_{\bf n}$ is a polynomial in $b$ and $c$ independent of $k$.
\end{prop}
The statement (a) is a partial result of Conjecture~14. In \cite{ERS24}, they conjectured that these inequalities are the equalities. Although the author have not shown it, this would be true.
The equality (\ref{eq: 15}) is the same as Conjecture~15. The equality (\ref{eq: 16-18}) is similar to Conjecture~16, 17, 18.
\begin{proof}
\textup{(a)} Let ${\bf n}=l{\bf n}_0$ with $l=\gcd(n_1,n_2)$ and ${\bf n}_0 \in N^{+}_{\mathrm{pr}}$. For any $s \in \mathbb{Z}_{\geq 0}$, since $\binom{x}{s}$ is a polynomial in $x$ with $\deg_{x}\binom{x}{s} = s$, we express
\begin{equation}
\binom{x}{s}=\sum_{t=0}^{s}\alpha_{t;s}x^{t}
\end{equation}
for some $\alpha_{t;s} \in \mathbb{Q}$.
Consider each term in (\ref{eq: u to tau}). For any $s_1,\dots,s_l$ with $s_1+2s_2+\cdots+ls_l=l$, we have
\begin{equation}
\begin{aligned}
&\ \prod_{j=1}^{l}\binom{gU_{j{\bf n}_0}(c,b)}{s_j}\\
=&\ \prod_{j=1}^{l}\sum_{t_j=0}^{s_j}\alpha_{t_j;s_j}\cdot(gU_{j{\bf n}_0}(c,b))^{t_j}\\
=&\ \sum_{\substack{0 \leq t_{i} \leq s_i, \forall i}}\alpha_{t_1;s_1}\cdots\alpha_{t_l;s_l}\left\{\prod_{j=1}^{l}(U_{j{\bf n}_0}(c,b))^{t_j}\right\}g^{t_1+\cdots+t_l}.
\end{aligned}
\end{equation}
Since $\tau^{(b,c)}({\bf n};k)$ is the coefficient of $g^{k}$, it may be expressed as a linear combinations of $\prod_{j=1}^{l}(U_{j{\bf n}_0}(c,b))^{t_j}$ such that $0\leq t_i \leq s_i$, $t_1+\cdots+t_l=k$, and $s_1+2s_2+\cdots+ls_l = l$. By Proposition~\ref{prop: 2}, we have
\begin{equation}
\begin{aligned}
&\ \deg_b\,\prod_{j=1}^{l}(U_{j{\bf n}_0}(c,b))^{t_j}\\
=&\ t_1(n'_2-1)+t_2(2n'_2-1)+\cdots+t_l(n'_l-1)\\
=&\ (t_1+2t_2+\cdots+lt_l)n'_2-(t_1+t_2+\cdots+t_l)\\
\leq&\ (s_1+2s_2+\cdots+ls_l)n'_2-k=ln'_2-k=n_2-k.
\end{aligned}
\end{equation}
Thus, we have $\deg_{b}\tau^{(b,c)}({\bf n};k) \leq n_2-k$. We may obtain $\deg_c \tau^{(b,c)}({\bf n}) \leq n_1-k$ by doing a similar argument.
\\
\textup{(b)}\ Since $\gcd(n_1,n_2)=1$, it holds that ${\bf n} \in N^{+}_{\mathrm{pr}}$. Thus, by Lemma~\ref{lem: u to tau}, we have $\tau^{(b,c)}({\bf n})=gU_{{\bf n}}(c,b)$ and $\tau^{(b,c)}({\bf n};1)=U_{\bf n}(c,b)$.
Consider
\begin{equation}\label{eq: expansion formula}
\tau^{(b,c)}(k{\bf n})=\sum_{\substack{
s_1,s_2,\dots,s_k \geq 0,\\
s_1+2s_2+\cdots+ks_k=k}}
\prod_{j=1}^{k}\binom{gU_{j{\bf n}}(c,b)}{s_j}.
\end{equation}
In the above sum, the term of $g^{k}$ appear if and only if $(s_1,s_2,\dots,s_k)=(k,0,\dots,0)$. Thus, $\tau^{(b,c)}(k{\bf n};k)$ is the same as the coefficient of $g^{k}$ in this term. We have
\begin{equation}\label{eq: binomial expansion}
\begin{aligned}
&\ \binom{gU_{{\bf n}}(c,b)}{k}\\
=&\ \frac{gU_{\bf n}(c,b)(gU_{\bf n}(c,b)-1)\cdots(gU_{\bf n}(b,c)-(k-1))}{k!}.
\end{aligned}
\end{equation}
The above numerator is a product of $k$-factors $gU_{\bf n}(c,b)-j$ ($j=0,1,\dots,k-1$). Thus, the coefficient of $g^k$ is 
\begin{equation}
\frac{(U_{\bf n}(c,b))^k}{k!}=\frac{(\tau^{(b,c)}({\bf n};1))^{k}}{k!}.
\end{equation}
Thus, (\ref{eq: 15}) holds.
\par
Next, we show (\ref{eq: 16-18}). Consider (\ref{eq: expansion formula}).
Then, a term $g^{k-1}$ appears only if $(s_1,s_2,\dots,s_k)$ is $(k,0,\dots,0)$ or $(k-2,1,0,\dots,0)$. The term corresponding to $(s_1,\dots,s_k)=(k,0,\dots,0)$ is in (\ref{eq: binomial expansion}). Since the numerator consists of the $k$-factors $gU_{\bf n}(c,b)-j$, the coefficient of $g^{k-1}$ comes form the following choice.
\begin{quote}
We choose $gU_{\bf n}(c,b)$ from $k-1$-factors, and we choose the number $-j$ from the other factor. 
\end{quote}
Thus, the coefficient of $g^{k-1}$ is
\begin{equation}
\begin{aligned}
-\frac{(1+2+\cdots+(k-1))}{k!}(U_{\bf n}(c,b))^{k-1}&=-\frac{1}{\cdot k!}\frac{(k-1)k}{2}(U_{\bf n}(c,b))^{k-1}\\
&=-\frac{1}{2\cdot(k-2)!}(U_{\bf n}(c,b))^{k-1}.
\end{aligned}
\end{equation}
The term corresponding to $(s_1,\dots,s_k)=(k-2,1,0,\dots,0)$ is
\begin{equation}
\begin{aligned}
&\ gU_{2{\bf n}}(c,b)\binom{gU_{\bf n}(c,b)}{k-2}
\\
=&\ \frac{(gU_{2{\bf n}}(c,b))(gU_{\bf n}(c,b))(gU_{\bf n}(c,b)-1)\cdots(gU_{\bf n}(b,c)-(k-3))}{(k-2)!}.
\end{aligned}
\end{equation}
The numerator consists of $k-1$ factors $gU_{2{\bf n}}(c,b)$ and $gU_{\bf n}(c,b)-j$ ($j=0,1,\dots,k-3$). Thus, the coefficient of $g^{k-1}$ is
\begin{equation}
\frac{1}{(k-2)!}(U_{\bf n}(c,b))^{k-2}U_{2{\bf n}}(c,b).
\end{equation}
Thus, we have
\begin{equation}
\begin{aligned}
&\ \tau^{(b,c)}(k{\bf n};k-1)\\
=&\ \frac{1}{(k-2)!}(U_{\bf n}(c,b))^{k-2}U_{2{\bf n}}(c,b)-\frac{1}{2\cdot(k-2)!}(U_{\bf n}(c,b))^{k-1}\\
=&\ \frac{(U_{\bf n}(c,b))^{k-2}}{(k-2)!}\left(U_{2{\bf n}}(c,b)-\frac{1}{2}U_{\bf n}(c,b)\right).
\end{aligned}
\end{equation}
Set $p_{\bf n}=U_{2{\bf n}}(c,b)-\frac{1}{2}U_{\bf n}(c,b)$. This is a polynomial independent of $k$. Moreover, since $\tau^{(b,c)}({\bf n};1)=U_{\bf n}(c,b)$, we have
\begin{equation}
\tau^{(b,c)}(k{\bf n};k-1)=\frac{(\tau^{(b,c)}({\bf n};1))^{k-2}}{(k-2)!}p_{\bf n}.
\end{equation}
\end{proof}

\bibliography{Some_coefficients_of_rank_2_cluster_scattering_diagrams}
\bibliographystyle{alpha}
\end{document}